\documentclass[letterpaper, 10 pt, conference]{ieeeconf}  

\IEEEoverridecommandlockouts                              
\usepackage{graphicx}
\usepackage{amsmath} 
\usepackage{amssymb}  
\usepackage{amsfonts}
\usepackage{graphicx}
\usepackage{algorithm, algorithmic}
\usepackage{subcaption}
\numberwithin{algorithm}{section}

\newtheorem{defn}{Definition}
\newtheorem{thm}{Theorem}
\newtheorem{cor}{Corollary}
\newtheorem{rem}{Remark}
\newtheorem{lem}{Lemma}

\title{\LARGE \bf
Exact and Efficient Hamilton-Jacobi Reachability for Decoupled Systems
}

\author{Mo~Chen and Claire~J.~Tomlin
\thanks{This work is supported in part by NSF under CPS:ActionWebs (CNS-0931843) and CPS:FORCES (CNS1239166), by NASA under grants NNX12AR18A and UCSCMCA-14-022 (UARC), by ONR under grants N00014-12-1-0609, N000141310341 (Embedded Humans MURI), and MIT\_5710002646 (SMARTS MURI), and by AFOSR under grants UPenn-FA9550-10-1-0567 (CHASE MURI) and the SURE project. The research of M. Chen has received funding from the ``NSERC PGS-D'' Program.}
\thanks{M.~Chen, and C.~J.~Tomlin are with the Department of Electrical Engineering and Computer Sciences,
        University of California, Berkeley, CA 94720, USA
        {\tt\small \{mochen72,tomlin\}@eecs.berkeley.edu}}
}

\begin{document}

\maketitle

\thispagestyle{empty}
\pagestyle{empty}

\begin{abstract}
Reachability analysis is important for studying optimal control problems and differential games, which are powerful theoretical tools for analyzing and modeling many practical problems in robotics, aircraft control, among other application areas. In reachability analysis, one is interested in computing the reachable set, defined as the set of states from which there exists a control, despite the worst disturbance, that can drive the system into a set of target states. The target states can be used to model either unsafe or desirable configurations, depending on the application. Many Hamilton-Jacobi formulations allow the computation of reachable sets; however, due to the exponential complexity scaling in computation time and space, problems involving approximately 5 dimensions become intractable. A number of methods that compute an approximate solution exist in the literature, but these methods trade off complexity for optimality. In this paper, we eliminate complexity-optimality trade-offs for time-invariant decoupled systems using a decoupled Hamilton-Jacobi formulation that enables the exact reconstruction of high dimensional solutions via low dimensional solutions of the decoupled subsystems. Our formulation is compatible with existing numerical tools, and we show the accuracy, computation benefits, and an application of our novel approach using two numerical examples.
\end{abstract}

\section{Introduction}
\label{sec:intro}
Optimal control problems and differential games have been extensively studied \cite{Mitchell05,Tomlin00, Basar99, Lygeros99}, and have received growing interest in the recent past. These powerful theoretical tools allow us to analyze a variety of real world problems, including path planning, collision avoidance, safety verification, among other applications in robotics, aircraft control, security, and other domains \cite{DotAF09,Madrigal09,Erzberger06,Tice91}.

In an optimal control problem, one aims to drive a controlled dynamical system into a set of states called the target set; depending on the application, the target set can model the set of either desirable or undesirable configurations. In a reachability framework, one aims to determine the backwards reachable set, defined as the set of states from which a control exists to drive the system into the target set. Differential games involve two adversarial players (Player 1 and Player 2). Player 2 seeks to drive a system to a target set, while Player 1 seeks to prevent Player 2 from doing so. One again aims to determine the backwards reachable set, which in this case is defined as the set of states from which a control from Player 2 exists to drive the system into the target set, despite the optimal adversarial control from Player 1.

Reachability is an effective way to analyze optimal control problems and differential games because it provides guarantees on system performance and safety. Reachability problems involving one player can be posed as a minimum (maximum) cost game where the player minimizes the minimum value over time of some cost function representing the proximity to the target set. In the case of a differential game, Player 1 maximizes the minimum cost over time, while Player 2 minimizes it. \cite{Mitchell05} has shown that the backwards reachable set can be obtained by solving a Hamilton-Jacobi Partial Differential Equation (HJ PDE) with a terminal condition specifying the target set. Many similar formulations of the backwards reachability problems also exist \cite{Fisac15,Bokanowski10, Barron90}. HJ reachability has been successfully used to solve problems such as aircraft collision avoidance \cite{Mitchell05}, automated in-flight refueling \cite{Ding08}, and reach-avoid games \cite{Chen14,Huang11}.

The techniques for computing backwards reachable sets via solving an HJ PDE are very flexible and can be applied to a large variety of system dynamics when the problem dimensionality is low. Furthermore, many numerical tools have been developed to solve these equations, making the HJ approach practically appealing \cite{LSToolbox, Mitchell08, Osher02, Sethian96}. For higher dimensional problems, various techniques such as those involving projections \cite{Mitchell11,Mitchell03}, approximate dynamic programming \cite{McGrew08}, and occupation measure theory \cite{Lasserre08} have been proposed. While these approximation techniques alleviate the computation complexity, they give up optimality and sometimes give overly conservative results.

This paper resolves the complexity-optimality trade-off for time-invariant systems with decoupled dynamics. We present a decoupled formulation of HJ reachability for decoupled systems, defined in \eqref{eq:dyn_decoupled}. By considering the decoupled component separately and solving lower dimensional HJ PDEs for each subsystem, we reduce the computation time and space complexity substantially. Our approach also \textit{exactly} recovers the solution to the original, high dimensional PDE.
\section{Problem Formulation}
\label{sec:formulation}
Consider a differential game between two players described by the time-invariant system
\begin{equation} \label{eq:dyn}
\dot{z} = f(z, u, d), \text{ almost every } t\in [-T,0],
\end{equation}

\noindent where $z\in\mathbb{R}^n$ is the system state, $u\in \mathcal{U}$ is the control of Player 1, and $d\in\mathcal{D}$ is the control of Player 2. We assume $f:\mathbb{R}^n\times \mathcal{U} \times \mathcal{D} \rightarrow \mathbb{R}^n$ is uniformly continuous, bounded, and Lipschitz continuous in $z$ for fixed $u,d$, and the control functions $u(\cdot)\in\mathbb{U},d(\cdot)\in\mathbb{D}$ are drawn from the set of measurable functions\footnote{A function $f:X\to Y$ between two measurable spaces $(X,\Sigma_X)$ and $(Y,\Sigma_Y)$ is said to be measurable if the preimage of a measurable set in $Y$ is a measurable set in $X$, that is: $\forall V\in\Sigma_Y, f^{-1}(V)\in\Sigma_X$, with $\Sigma_X,\Sigma_Y$ $\sigma$-algebras on $X$,$Y$.}. As in \cite{Mitchell05,Evans84,Varaiya67}, we allow Player 2 to only use nonanticipative strategies $\gamma$, defined by

\begin{equation}
\begin{aligned}
\gamma \in \Gamma &:= \{\mathcal{N}: \mathbb{U} \rightarrow \mathbb{D} \mid  \\
& u(r) = \hat{u}(r) \text{ for almost every } r\in[t,s] \\
&\Rightarrow \mathcal{N}[u](r) = \mathcal{N}[\hat{u}](r) \text{ for almost every } r\in[t,s]\}
\end{aligned}
\end{equation}

We further assume that the system is a \textit{decoupled system}.
\begin{defn}
\textbf{Decoupled system}. A system \eqref{eq:dyn} is a decoupled system if it can be split into $N$ components, denoted $\{x_i\}_{i=1}^N$ where $z = (x_1, \ldots, x_N)$, that satisfy the following:

\begin{equation} \label{eq:dyn_decoupled}
\begin{aligned}
\dot{x_i} &= f_i(x_i, u_i, d_i), \text{ almost every } t\in [-T,0], \\
i &= 1,\ldots, N,
\end{aligned}
\end{equation}

\noindent where $x_i\in\mathbb{R}^{n_i}$ is $i$th component of the full state, $u_i\in \mathcal{U}_i$ is $i$th component of the control of Player 1, and $d_i\in\mathcal{D}_i$ is $i$th component of the control of Player 2. Based on this assumption and the assumptions on $f(\cdot,\cdot,\cdot),u(\cdot),d(\cdot)$, we have that $f_i:\mathbb{R}^{n_i}\times \mathcal{U}_i \times \mathcal{D}_i \rightarrow \mathbb{R}^{n_i}$ is uniformly continuous, bounded, and Lipschitz continuous in $x_i$ for fixed $u_i,d_i$ and $u_i(\cdot),d_i(\cdot)$ are measurable. Note that $\sum_{i=1}^N n_i = n$. 
\end{defn}

Denote system trajectories, which are solutions to \eqref{eq:dyn}, as
\begin{equation}
\xi_f(s; z, t, u(\cdot), d(\cdot)): [t,0]\rightarrow \mathbb{R}^n.
\end{equation}

$\xi_f$ satisfies initial conditions $\xi_f(t; z, t, u(\cdot), d(\cdot))=z$ and the following differential equation almost everywhere
\begin{equation}
\frac{d}{ds}\xi_f(s; z, t, u(\cdot), d(\cdot)) = f(\xi_f(s; z, t, u(\cdot), d(\cdot)), u(s), d(s))
\end{equation}

In our differential game, the goal of Player 2 is to drive the system into some target set $\mathcal{L}$, and the goal of Player 1 is to drive the system away from it. The set $\mathcal{L}$ is represented as the zero sublevel set of a bounded, Lipschitz continuous function $l:\mathbb{R}^n\rightarrow\mathbb{R}$, $\mathcal{L}=\{z\in\mathbb{R}^n \mid l(z)\le 0\}$.

Such a function always exists, since we can choose $l(\cdot)$ to be a signed distance function; we call $l(\cdot)$ the implicit surface function representing the set $\mathcal L$. In accordance with our decoupled dynamics, we assume that $l$ can be represented as a maximum of $N$ bounded, Lipschitz continuous functions $l_i:\mathbb{R}^{n_i}\rightarrow \mathbb{R}$, $l(z) = l(x_1,\ldots,x_N)=\max_i l_i(x_i)$ where $l_i(x_i)$ are implicit surface functions representing $\mathcal{L}_i$ so that $z\in\mathcal{L} \Leftrightarrow x_i\in\mathcal{L}_i \ \forall i$. Note that with the definition of $l(z)$ and $l_i(x_i)$, we have that 

\begin{equation} \label{eq:decoupled_target}
\mathcal{L} = \bigcap_i \mathcal L_i.
\end{equation}

Given the decoupled system \eqref{eq:dyn_decoupled} and the target set $\mathcal{L}$ in the form \eqref{eq:decoupled_target} represented by $l(\cdot)$, our goal in this paper is to compute the backwards reachable set, $\mathcal{V}(t)$, in the low-dimensional space $\mathbb{R}^{n_i}$ of each of the decoupled components $x_i$ as opposed to in the full system state space $\mathbb{R}^n$. $\mathcal{V}(t)$ is defined as

\begin{equation}
\begin{aligned}
\mathcal{V}(t) &:= \{z\in\mathbb{R}^n \mid \exists \gamma\in\Gamma \text{ such that} \\
&\forall u(\cdot)\in\mathbb{U}, \exists s \in [t,0], \xi_f(s; z, t, u(\cdot), \gamma[u](\cdot) \in \mathcal{L}) \}
\end{aligned}
\end{equation}

\begin{rem} One may have noticed that if $\mathcal{L} = \bigcup_i \mathcal L_i$, one would be able to simply find $\mathcal V_i(t)$ by solving \eqref{eq:HJIPDE} with $\mathcal L_i$ as the target set, and then obtain $\mathcal V(t) = \bigcup_i \mathcal V_i(t)$. However, it is crucial to observe that we are interested in the case where $\mathcal{L} = \bigcap_i \mathcal L_i$, in which a simple union of $\mathcal V_i(t)$ would \textit{not} yield the correct reachable set $\mathcal V(t)$.
\end{rem}
\section{Solution} \label{sec:solution}
\subsection{HJ Reachability: Full Formulation} \label{subsec:full}
In \cite{Mitchell05}, the authors showed that the backwards reachable set $\mathcal{V}(t)$ can be obtained as the zero sublevel set of the viscosity solution \cite{Crandall83} $V(t,z)$ of the following terminal value HJ PDE:

\begin{equation} \label{eq:HJIPDE}
\begin{aligned}
D_t V(t,z) + \min \{0, \max_{u\in\mathcal{U}} \min_{d\in\mathcal{D}} D_z V(t,z) \cdot f(z,u,d)]\} &= 0 \\
V(0,z) = l(z)&
\end{aligned}
\end{equation}

\noindent from which we obtain $\mathcal{V}(t) = \{z\in\mathbb{R}^n \mid V(t,z)\le 0\}$ from the bounded, Lipschitz function $V(t,z)$ that is also continuous in both $z$ and $t$ \cite{Evans84}.

\cite{Mitchell05} and similar approaches, such as \cite{Fisac15,Bokanowski10, Barron90}, are compatible with well-established numerical methods \cite{LSToolbox, Mitchell08, Osher02, Sethian96}. However, these approaches become intractable quickly as the dimensionality of the problem $n$ increases. Numerically, the solution $V(t,z)$ is computed on a grid, and the number of grid points increases exponentially with the number of dimensions.

Decoupled dynamics allow for tractable or faster computation of reachable sets in the individual decoupled components. Some authors \cite{Mitchell11,Mitchell03} have proposed methods for combining or stitching together these reachable set components into the full reachable set. These methods work reasonably well, but introduce conservatism in various ways. In the next subsections, we will provide a method for combining solutions to the lower dimensional HJ PDEs to construct the \textit{exact} full solution in the original HJ PDE.

\subsection{HJ Reachability: Decoupled Formulation} \label{subsec:decoupled}
Observe that \eqref{eq:HJIPDE} can be viewed as an equation involving two cases. Depending on which of the arguments in the outer-most minimum is active, $\eqref{eq:HJIPDE}$ becomes one of \eqref{eq:HJIPDE_0} or \eqref{eq:HJIPDE_H}:

\begin{equation} \label{eq:HJIPDE_0}
D_t V(t,z) = 0
\end{equation}

\begin{equation} \label{eq:HJIPDE_H}
D_t V(t,z) + \max_{u\in\mathcal{U}} \min_{d\in\mathcal{D}} D_z V(t,z) \cdot f(z,u,d)= 0
\end{equation}

This motivates us to define the following sets $\mathcal{F}_1(t), \mathcal{F}_2(t)$ which characterize which of the outer-most minimum operation is active in \eqref{eq:HJIPDE}.

\begin{equation} \label{eq:freeze_regions}
\begin{aligned}
\mathcal{F}_1(t) &= \{z\in\mathbb{R}^n \mid \max_{u\in\mathcal{U}} \min_{d\in\mathcal{D}} D_z V(t,z) \cdot f(z,u,d) > 0\} \\
\mathcal{F}_2(t) &= \{z\in\mathbb{R}^n \mid \max_{u\in\mathcal{U}} \min_{d\in\mathcal{D}} D_z V(t,z) \cdot f(z,u,d) \le 0\}
\end{aligned}
\end{equation}

Note that $\mathcal{F}_1(t)$ is the complement of $\mathcal{F}_2(t)$, $\mathcal F_1(t) = \mathcal{F}^C_2(t)$, for all time. At a given $t$, in $\mathcal{F}_1(t)$, $V(t,z)$ satisfies \eqref{eq:HJIPDE_0}; in $\mathcal{F}_2(t)$, $V(t,z)$ satisfies \eqref{eq:HJIPDE_H}. We now show an important property of $\mathcal{F}_1(t)$ and $\mathcal{F}_2(t)$ in the Lemma and Corollary below. These will be used to show that our proposed decoupled formulation allows exact computation of $V(t,z)$, by computation of $V_i(t,x_i)$, value functions of lower dimensional spaces.

\begin{lem} \label{lem:freeze_regions}
$z\in\mathcal{F}_1(-t_0) \Rightarrow z\in\mathcal{F}_1(t) \ \forall t \in [-T, -t_0]$ for some $t_0$ such that $0<t_0<T$. 
\end{lem}

\begin{proof}
Suppose $z\in\mathcal{F}_1(-t_0)$, then by \eqref{eq:HJIPDE_0} we have the following: 

\begin{itemize}
\item $D_t V(-t_0,z) = 0$. Thus, $V(t,z)$ becomes independent of $t$ at $t=-t_0$.
\item since $\mathcal{F}_1$ is an open set, there exists a neighborhood around $z$ that is contained in $\mathcal{F}_1$. Thus $V(t,z)$ is also independent of $t$ in a neighborhood of $z$.
\item By \eqref{eq:freeze_regions}, we have $\max_{u\in\mathcal{U}} \min_{d\in\mathcal{D}} D_z V(-t_0,z) \cdot f(z,u,d) > 0$. 
\end{itemize}

Let $t_1\in (t_0, T]$ and suppose $z\in\mathcal{F}_2$ at $t=-t_1$. Then, by \eqref{eq:freeze_regions}, $\max_{u\in\mathcal{U}} \min_{d\in\mathcal{D}} D_z V(-t_1,z) \cdot f(z,u,d) \le 0$. Since $f$ is independent of $t$, this necessarily means that $D_z V(-t_1,z) \neq D_z V(-t_0,z)$.

This implies $\exists z_0$ in a neighborhood of $z$ such that $V(-t_1,z_0) \neq V(-t_0,z_0)$.

\begin{equation}
V(-t_1,z_0) - V(-t_0,z_0) =\Delta \neq 0
\end{equation}

However, by \eqref{eq:HJIPDE_0}, $V(t,z_0) = V(-t_0,z_0) \ \forall t\in (-t_1, -t_0]$. In particular, then, we have for any $\epsilon>0$

\begin{equation}
\frac{V(-t_1,z_0) - V(-t_1+\epsilon,z_0)}{\epsilon} = \frac{\Delta}{\epsilon}.
\end{equation}

This means that $\forall M\in\mathbb{R}$, $\exists \epsilon>0$ such that 

\begin{equation}
\frac{V(-t_1,z_0) - V(-t_1+\epsilon,z_0)}{\epsilon} > M,
\end{equation}

\noindent which is a contradiction since $V(t,z_0)$ is Lipschitz continuous. Therefore, since $\max_{u\in\mathcal{U}} \min_{d\in\mathcal{D}} D_z V(t,z) \cdot f(z,u,d) \ge 0 \ \forall t\le -t_0$, we have that $z\in\mathcal{F}_1(t) \ \forall t \in [-T, -t_0]$ by \eqref{eq:freeze_regions}.

%
%
%
\end{proof}

\begin{cor} \label{cor:freeze_regions}
$z\in\mathcal{F}_2(-t_1) \Rightarrow z\in\mathcal{F}_2(t) \ \forall t\in [-t_1,0]$.
\end{cor}

\begin{proof}
Suppose $\exists t_0\in [0,t_1), z\notin\mathcal{F}_2(-t_0)$ but $z\in\mathcal{F}_2(-t_1)$. Since $\mathcal{F}_1(t)$ is the complement of $\mathcal{F}^C_2(t)$, this implies $z\in\mathcal{F}_1(-t_0)$. By Lemma \ref{lem:freeze_regions}, we must have that $z\in\mathcal{F}_1(-t_1)$ since $t_1\in[-T, -t_0]$, a contradiction.
\end{proof}

We can now state our main theorem.

\begin{thm} \label{thm:decoupled}
The solution to \eqref{eq:HJIPDE} for a decoupled system with dynamics \eqref{eq:dyn_decoupled} and terminal condition $l(z)=\max_i l_i(x_i)$ is given by

\begin{equation}
\begin{aligned}
V(t,z) &= \max_i V_i(\bar t(z), x_i) &\forall z \in \mathcal F_1(t) \\
V(t,z) &= \max_i V_i(t, x_i) &\forall z \in \mathcal F_2(t)
\end{aligned}
\end{equation}

\noindent where $V_i(t,x_i), i=1\ldots,N$ are the viscosity solutions to

\begin{equation} \label{eq:HJIPDE_Vi}
\begin{aligned}
D_t V_i(t,x_i) + \max_{u_i\in\mathcal{U}_i} \min_{d_i\in\mathcal{D}_i} D_{x_i} V_i(t,x_i) \cdot f_i(x_i,u_i,d_i) &= 0 \\
V_i(0,x_i) = l_i(x_i), &
\end{aligned}
\end{equation}

\noindent and $\bar t(z)$ is the smallest time such that $z\in \mathcal F_1(t)$, i.e.

\begin{equation}
\bar t(z) = \inf_{\tau>t} \{z\in \mathcal F_1 (\tau)\}
\end{equation}
\end{thm}

\begin{proof}
\textbf{Case 1}: By Lemma \ref{lem:freeze_regions}, we have $\forall z\in \mathcal F_1(t), z\in \mathcal F_1(\tau) \forall \tau \le \bar t(z)$. Therefore, $V(t,z)$ satisfies \eqref{eq:HJIPDE_0} $\forall t\le \bar t(z)$, so $V(t,z) = V(\bar t(z), z)$. Case 2 of this proof would then imply $V(\bar t(z), z) = \max_i V_i(\bar t(z), x_i)$.

\textbf{Case 2}: Consider a target set represented by the zero sublevel set of the function $l(z)$, where $l(z)=\max_i l_i(x_i)$. By \eqref{eq:HJIPDE}, we have that $V(0,z) = \max_i l_i(x_i)$. Define functions $V_i(t,x_i)$ such that $V_i(0,x_i)=l_i(x_i)$, then at $t=0$, we have $V(t,z) = \max_i V_i(t,x_i)$, and

\begin{equation}
\begin{aligned}
D_t V(t,z) &= \sum_i 1\{i = \arg \max_i V_i(t,x_i) \}D_t V_i(t,x_i) \\
D_z V(t,z) &= \sum_i 1\{i = \arg \max_i V_i(t,x_i) \} I_{x_i} D_{x_i} V_i(t,x_i) \\
\end{aligned}
\end{equation}

\noindent where $1\{\cdot\}$ is the indicator function that is $1$ when its argument is true and $0$ otherwise, and $I_{x_i}$ is an matrix in $\mathbb R^{n\times n_i}$ of all zeros except for in the rows corresponding to the $x_i$ component where it is the identity matrix in $\mathbb{R}^{n_i\times n_i}$.

Now, consider all points $z\in\mathcal{F}_2(t)$, in which $V(t,z)$ satisfies \eqref{eq:HJIPDE_H}. Substituting $V(t,z) = \max_i V_i(t,x_i)$ into \eqref{eq:HJIPDE_H}, we have

\begin{equation} \label{eq:HJIPDE_D}
\begin{aligned}
\sum_i 1\{i = \arg \max_i V_i(t,x_i) \}\big[D_t V_i(t,x_i) +\\
 \max_{u_i\in\mathcal{U}_i} \min_{d_i\in\mathcal{D}_i} D_{x_i} V_i(t,x_i) \cdot f_i(x_i,u_i,d_i)\big] &= 0.
\end{aligned}
\end{equation}

Equation \eqref{eq:HJIPDE_D} states that in the region where $V_i(t,x_i)$ is the maximum among $\{V_j(t,x_j)\}_{j=0}^N$, we have $V(t,z) = V_i(t,x_i)$, where $V_i(t,x_i)$ satisfies \eqref{eq:HJIPDE_Vi}.

Consider auxiliary functions $W_i(t,x_i)$ which satisfy, for all $t\in[-T,0]$ and all $x_i$,
\begin{equation} \label{eq:HJIPDE_Wi}
\begin{aligned}
D_t W_i(t,x_i) + \max_{u_i\in\mathcal{U}_i} \min_{d_i\in\mathcal{D}_i} D_{x_i} W_i(t,x_i) \cdot f_i(x_i,u_i,d_i) &= 0 \\
W_i(0,x_i) = l_i(x_i). &
\end{aligned}
\end{equation}

By Corollary \ref{cor:freeze_regions}, we have that $V_i(t,x_i)$ and $W_i(t,x_i)$ both satisfy the same PDE with the same terminal conditions, $i=1\ldots,N$. Therefore, $V_i(t,x_i)=W_i(t,x_i),\forall i=1,\ldots,N$.
\end{proof}

\subsection{Decoupled Formulation Algorithm} \label{subsec:alg}
Algorithmically, Theorem \ref{thm:decoupled} states the following:
\begin{enumerate}
\item $D_t V(t_0,z) = 0$ for some $t_0 \Rightarrow D_t V(t,z) = 0 \ \forall t\in[-T,t_0]$.
\item $\forall z\in \mathcal{F}_1(t), D_t V(t,z) = 0$.
\item $\forall z\in \mathcal{F}_2(t), V(t,z) = \max_i V_i(t,x_i)$ where $V_i(t,x_i)$ satisfies \eqref{eq:HJIPDE_Vi}.
\end{enumerate}

This gives us an efficient way to computed $V(t,z)$ by computing $V_i(t, x_i), i=1,\ldots,N$, transforming the original $n$-dimensional problem of computing $V(t,z)$ into the $N$ $n_i$-dimensional problems of computing $V_i(t,x_i),i=1\ldots,N$. Based on the conclusions we drew, the following algorithm \textit{exactly} computes $V(t,z)$, which satisfies \eqref{eq:HJIPDE}, with the above-mentioned computation benefits:

\begin{enumerate}
\item Initialize $V_i(t,x_i)=l_i(x_i), i=1\ldots,N$.
\item Compute $V_i(t,x_i), i=1\ldots,N$, by solving \eqref{eq:HJIPDE_Vi}.
\item Initialize $V(t,z)=l(t,z)=\max_i V_i(t,x_i)$.
\item Decrement $t$ from 0 to $-T$; for each time step $\bar{t}$:
\begin{enumerate}
\item Set the auxiliary variable 
\begin{equation*}
\tilde{V}(\bar{t},z) \leftarrow \max_i V_i(\bar{t}, x_i)
\end{equation*}
This step correctly computes $V(\bar{t},z)$ to be $\tilde{V}(\bar{t},z)$ for all $z\in\mathcal{F}_2(\bar{t})$.
\item Update the value function 

\begin{equation*}
V(\bar{t},z) \leftarrow \min\{V(\bar{t},z), \tilde{V}(\bar{t},z)\}
\end{equation*}

This step correctly computes $V(\bar{t},z)$ to satisfy $D_t V(t,z) = 0$ for all $z\in\mathcal{F}_1(\bar{t})$.
\end{enumerate}
\end{enumerate}

\subsection{Computation Time and Space Complexity Comparison}
For a state space discretization of $k$ grid points in each dimension, the computation time complexity decreases from $O(k^n)$ for the original problem in $\mathbb{R}^n$, to $O(\sum_i k^{n_i}) = O(k^{\max_i n_i})$ for the $N$ subproblems in $\mathbb{R}^{n_i}$. This is a computation speed improvement of many orders of magnitude. 

Directly solving \eqref{eq:HJIPDE} on a computational domain $\mathcal{S}\subset \mathbb{R}^n$ has a space complexity of $O(\tau k^n)$, where $\tau$ is the number of time steps of $V(t,z)$ being stored, since we need to store an $n$-dimensional grid for each of the $\tau$ time steps. The algorithm presented in \ref{subsec:alg} involves computing $V_i(t,x_i)$ on computation domains $\mathcal{X}_i\subset\mathbb{R}^{n_i},i=1,\ldots,N$. Each $V_i(t,x_i)$ thus has a space complexity of $O(\tau k^{n_i})$, making the overall space complexity $O(\tau k^{\max_i n_i})$.

From $V_i(t,x_i)$, we can then reconstruct $V(t,z)$ in any domain $\mathcal{Z} \subset \mathcal{X}_1 \times \mathcal{X}_2 \times \ldots \times \mathcal{X}_N$. Thus, by choosing $\mathcal{Z}$ to be a small subset of $\mathcal{X}_1 \times \ldots \times \mathcal{X}_N$, we can always avoid additional space complexity. This allows us to access $V(t,z), z\in\mathcal{Z}$. In practice, one would choose $\mathcal{Z}$ to be in a small region around a state $z$ of interest (eg. the current system state), and access the value function $V(t,z)$ as well as its gradient $D_z(t,z)$ at $z$; this allows one to determine of whether $z$ is in the reachable set based on the sign of $V(t,z)$, and compute the optimal controls $u(t), d(t)$ for both Player 1 and Player 2 respectively based on $D_z V(t,z)$ as follows:

\begin{equation} \label{eq:ctrl_syn}
\begin{aligned}
u^* &= \arg \max_{u\in\mathcal{U}} \min_{d\in\mathcal{D}} D_z V(t,z) \cdot f(z,u,d) \\
d^* &= \arg \min_{d\in\mathcal{D}} D_z V(t,z) \cdot f(z,u^*,d)
\end{aligned}
\end{equation}

Note that we do \textit{not} need to store $V(t,z)$ for all $z \in \mathcal{S} = \mathcal{X}_1 \times \ldots \times \mathcal{X}_N$. In fact, in many situations, storing $V(t,z)$ in the entire $\mathcal{S}$ is infeasible since the space complexity is exponential with the dimension of $\mathcal{S}$. With that caveat, we restate our algorithm for computing $V(t,z)$ from $V_i(t,x_i), i=1,\ldots,N$, explicitly noting memory allocation, to show that one only needs to store $V(t,z)$ for $z\in\mathcal{Z}\subset\mathcal{S}$, where $\mathcal{Z}$ is a very small subset of $\mathcal{S}$:

\begin{enumerate}
\item Initialize $V_i(t,x_i)=l_i(t,x_i)$ for $x_i\in\mathcal{X}_i, i = 1,\ldots,N$.
\item Compute $V_i(t,x_i)$ in $\mathcal{X}_i,i=1\ldots,N$ by solving \ref{eq:HJIPDE_Vi}.
\item Initialize $V(t,z)=l(t,z)$ in a small computation domain $\mathcal{Z}\subset\mathcal{X}_1\times\ldots\times\mathcal{X}_N$.
\item Decrement $t$ from 0 to $-T$; for each time step $\bar{t}$, perform the following computations in $\mathcal{Z}$:
\begin{enumerate}
\item $\tilde{V}(\bar{t},z) \leftarrow \max_i V_i(\bar{t}, x_i)$
\item $V(\bar{t},z) \leftarrow \min\{V(\bar{t},z), \tilde{V}(\bar{t},z)\}$
\end{enumerate}
\end{enumerate}

\subsection{Numerical Implementation}
Our proposed decoupled formulation involves solving \eqref{eq:HJIPDE_Vi} for each of the $N$ subsystems. As already mentioned, many numerical tools already exist for solving \eqref{eq:HJIPDE_Vi}; we will use the implementation in \cite{LSToolbox}. For the examples in this paper, we used the numerical schemes below.

For the numerical Hamiltonian $H(D_{x_i}V_i, x_i)=\max_{u\in\mathcal{U}} \min_{d\in\mathcal{D}} D_z V(t,z) \cdot f(z,u,d)$ in \eqref{eq:HJIPDE_Vi}, we used the Lax-Friedrich approximation \cite{Osher91}. For the numerical spatial derivatives $D_{x_i} V_i(t,x_i)$, we used a fifth-order accurate weighted essentially non-oscillatory scheme \cite{Osher91,Osher03}. For numerical time derivatives $D_t V_i(t,x_i)$, we used a third-order accurate total variation diminishing Runge-Kutta scheme \cite{Osher03, Shu88}. 

Computations were done on a computer with an Intel Core i7-2600K CPU running at 3.4 GHz, with 16 GB of memory.
\section{4D Quadrotor Collision Avoidance}
\label{sec:results}
Consider a simple quadrotor model consisting of two decoupled double-integrators:

\begin{equation}
\begin{aligned}
\dot{p}_x &= v_x, \qquad \dot{p}_y = v_y  \\
\dot{v}_x &= u_x, \qquad \dot{v}_y = u_y  \\
\underline{u} &\le |u_x|,|u_y| \le \bar{u}
\end{aligned}
\end{equation}

$p_x,p_y$ denote the $x$- and $y$-position of the quadrotor, and $v_x,v_y$ denote the $x$- and $y$-velocity. The control signals $u_x,u_y$ are the $x$- and $y$-acceleration of the quadrotor, constrained to be between $\underline{u}$ and $\bar{u}$.

Now, consider two quadrotors in a pursuit-evasion game, in which the evader (Player 1) aims to avoid collision, while the pursuer (Player 2) aims to cause a collision. The relative coordinates of the two quadrotors are given by the following state variables:

\begin{equation}
\begin{aligned}
p_{x,r} &= p_{x,i} - p_{x,j}, \qquad p_{y,r} = p_{y,1} - p_{y,2}\\
v_{x,r} &= v_{x,i} - v_{x,j}, \qquad v_{y,r} = v_{y,1} - v_{y,2}
\end{aligned}
\end{equation}

Given the above relative state variables, the relative dynamics of the two quadrotors are given by

\begin{equation}
\begin{aligned}
\dot{p}_{x,r}& = v_{x,r}, &\dot{p}_{y,r} &= v_{y,r} \\
\dot{v}_{x,r}& = u_{x,1} - u_{x,2}, &\dot{v}_{y,r} &= u_{y,1} - u_{y,2}\\
\end{aligned}
\end{equation}

Note that this system is decoupled, with $x_1 = (p_{x,r}, v_{x,r})\in\mathbb{R}^2$ as the first decoupled component, and $x_2 = (p_{y,r}, v_{y,r})\in\mathbb{R}^2$ as the second decoupled component. In the relative coordinates $z:=(p_{x,r}, v_{x,r}, p_{y,r}, v_{y,r})\in\mathbb{R}^4$ of the two quadrotors, we define the collision set of size $1$, representing the configurations in which the two quadrotors are considered to have collided, as the following set:

\begin{equation}
\mathcal{L} = \{z \in \mathbb{R}^4 \mid |p_{x,r}|, |p_{y,r}|\le 1\}
\end{equation}

\noindent with the corresponding implicit surface function $l(z)$ where $l(z)\le 0 \Leftrightarrow z\in\mathcal{L}$. Since we have a decoupled system, let $\mathcal{L}_i, i = 1,2$ be the following sets:

\begin{equation}
\begin{aligned}
\mathcal{L}_1 &= \{x_1 \in \mathbb{R}^2 \mid |p_{x,r}|\le 1\} \\
\mathcal{L}_2 &= \{x_2 \in \mathbb{R}^2 \mid |p_{y,r}|\le 1\}
\end{aligned}
\end{equation}

\noindent with corresponding implicit surface functions $l_i(x_i),i=1,2$. Then, we have $\mathcal{L} = \mathcal{L}_1 \cap \mathcal{L}_2$ and $l(z) = \max_i l_i(x_i),i=1,2$.

We will set $\mathcal{L}$ as the target set in our reachability problem, and compute the backwards reachable set $\mathcal{V}(t)$ from $\mathcal{L}$ using three methods:
\begin{itemize}
\item Solve \eqref{eq:HJIPDE} directly in $\mathbb{R}^4$ to obtain $V(t,z)$, whose zero sublevel set represents $\mathcal{V}(t)$.
\item Solve \eqref{eq:HJIPDE_Vi} in $\mathbb{R}^2$, $i=1,2$, to obtain $V(t,z)$ using our proposed decoupled formulation described in Section \ref{subsec:decoupled}.
\item Compute the analytic boundary of the reachable set.
\end{itemize}

For comparison purposes, for the first two methods we will compute $V(t,z)$ on the computation domain $[-5, 5]^4$. However, it is important to recall that for our proposed method described in Section \ref{subsec:decoupled}, we can significantly reduce space complexity by only storing a small part of $V(t,z)$.

\subsection{Reachable Set}
Since the state space of our system is 4D, we visualize various $(v_{x,r},v_{y,r})$ slices of the reachable set, whose boundary is given by $\{z\mid V(t=1.5,z)=0\}$. Figures \ref{fig:reach} shows these slices. The reachable set boundary computed using our proposed decoupled method is very close to the reachable set boundary computed by solving the full PDE \eqref{eq:HJIPDE} in $\mathbb{R}^4$ and to analytic reachable set boundary. Figures \ref{fig:reach_zoom} zooms in on the plots for a closer look.

\begin{figure}
	\centering
	\includegraphics[width=0.35\textwidth]{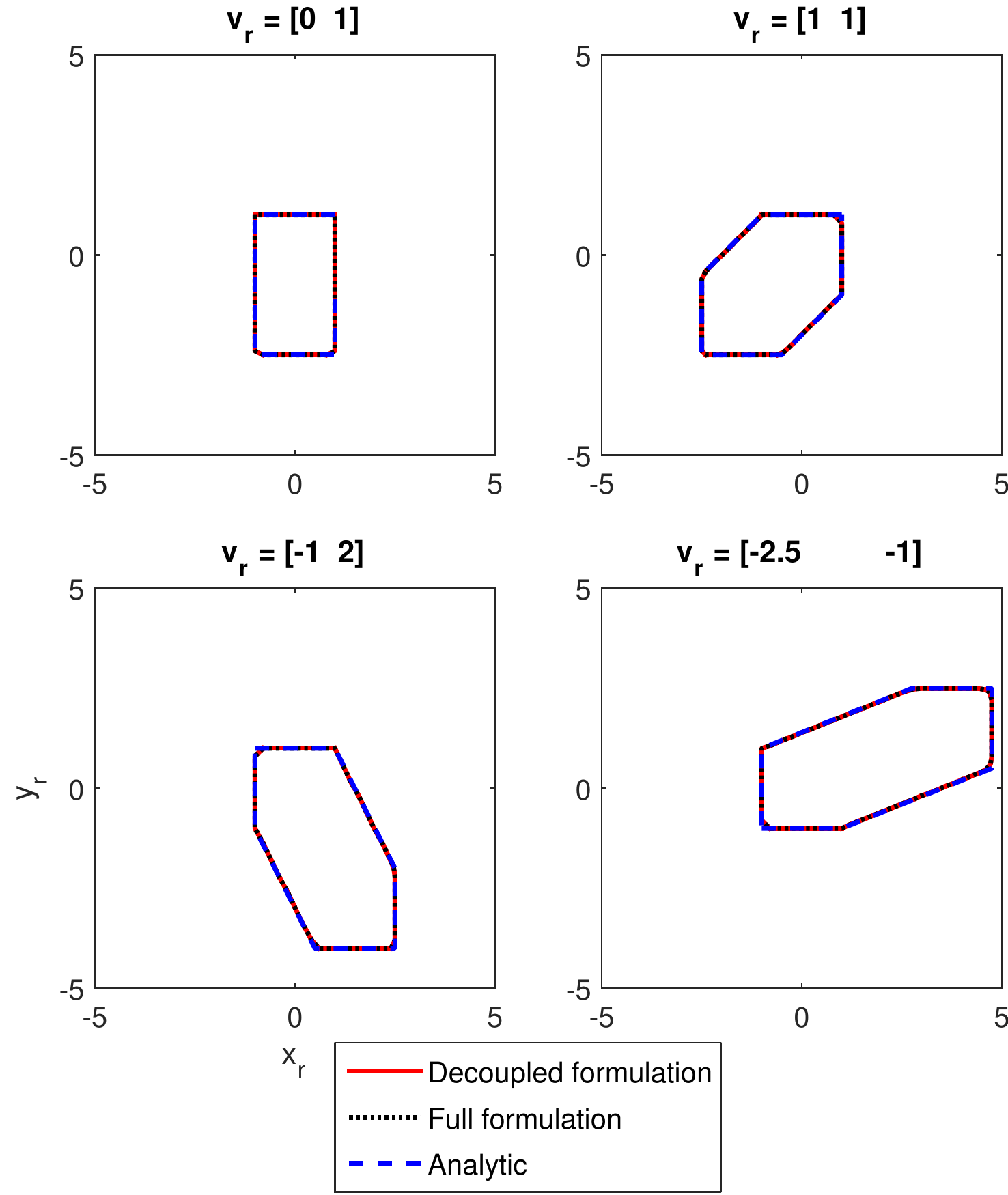}
	\caption{Various $(v_{x,r},v_{y,r})$ slices of the reachable set.}
	\label{fig:reach}
\end{figure}

\begin{figure}
	\centering
	\includegraphics[width=0.35\textwidth, trim=0 0 150 250, clip = true]{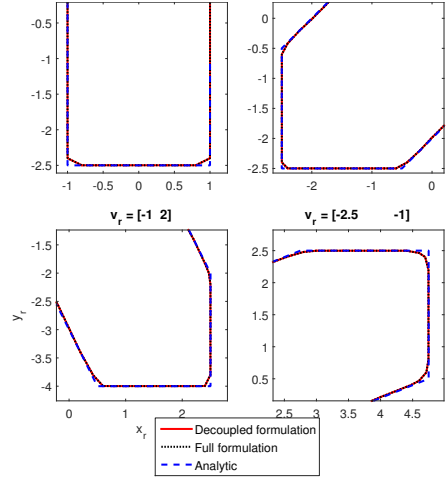}
	\caption{Various $(v_{x,r},v_{y,r})$ slices of the reachable set, zoomed in.}
	\label{fig:reach_zoom}
\end{figure}

\subsection{Performance}
In order to quantify the computation error, we converted $V(t,z)$ into a signed distance function $V_{sd}(t,z)$ from the boundary $V(t,z)=0$. This operation was first proposed in \cite{Chopp93} and can be done by solving the reinitialization PDE formulated in \cite{Sussman98}; for this operation, we use the implementation in \cite{LSToolbox}. We then evaluated approximately 24 million analytically-computed reachable set boundary points on the $V_{sd}(t,z)$; the resulting values represent how far each of the analytically-computed points are from the numerically-computed boundary. The values of $V_{sd}(t,z)$ on analytic boundary points are defined as the computation error. 

Figure \ref{fig:convergence} shows the error as a function of grid spacing. In terms of the maximum error (red curve), the decoupled formulation results in a numerically-computed reachable set boundary that is accurate within the size of the grid spacing (black line). On average, the error is approximately an order of magnitude smaller than the size of grid spacing (blue curve). Furthermore, we can see numerical convergence to the analytic solution as the grid spacing size decreases.

Figure \ref{fig:compTime} shows the computation time as a function of the number of grid points in each dimension. Here, we can see that the decoupled formulation is orders of magnitude faster than the full formulation, and can be done with many more grid points in each dimension. Lastly, the slopes of curves in the log-log plot show an $O(k^4)$ time complexity for the full formulation, and only $O(k^2)$ for the decoupled formulation. 

For the decoupled formulation, when we reconstruct the full value function in 4D (blue curve), the computation time hardly increases compared to when we do not perform full reconstruction (green curve). However, in general, we recommend that the value function in only a region near a state of interest should be computed. Without full reconstruction of the value function, we are able to obtain results with many more grid points (green curve), improving the accuracy of the numerical computation.

\begin{figure}
	\centering
	\includegraphics[width=0.25\textwidth]{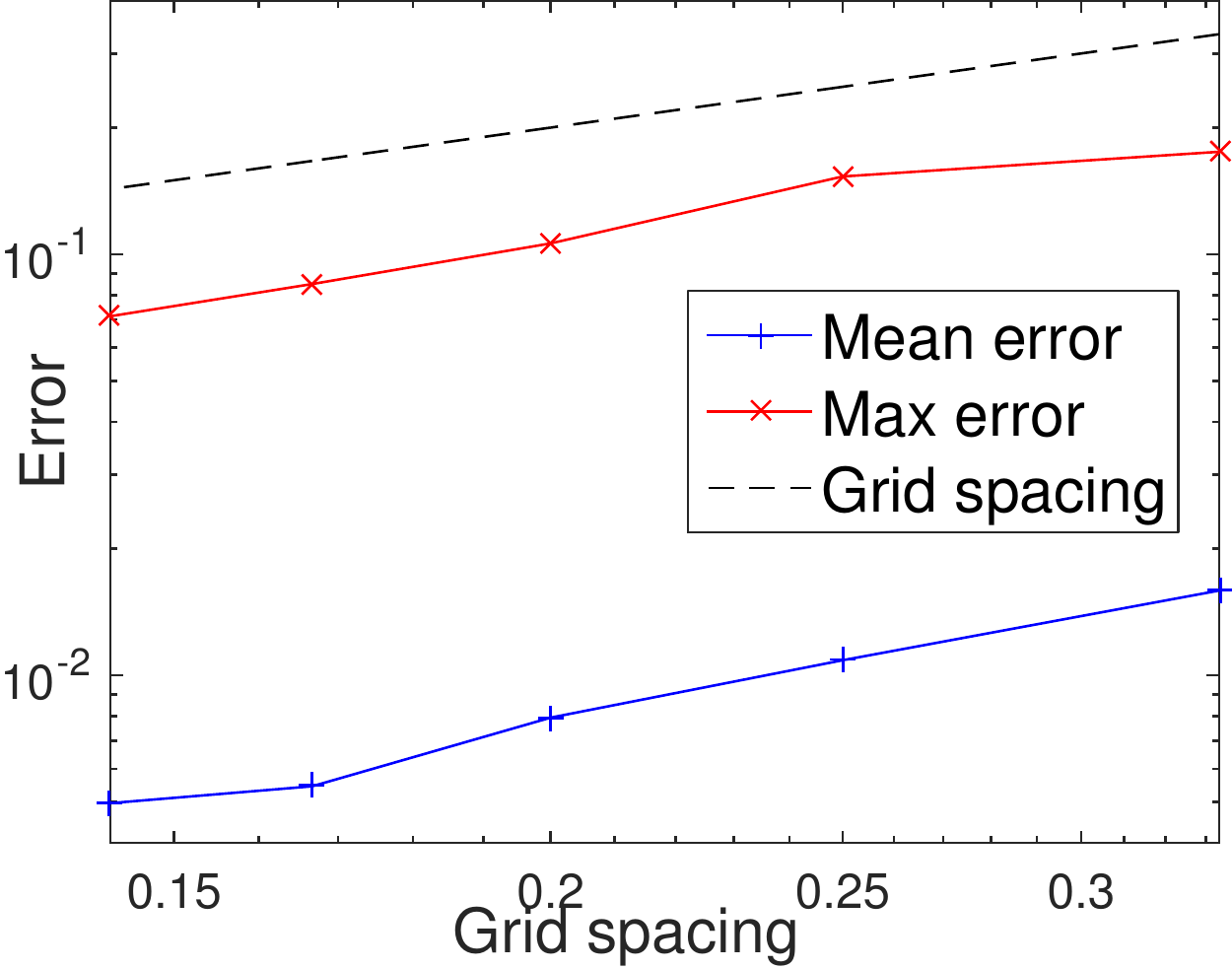}
	\caption{Mean and maximum error of the reachable set computed using the decoupled formulation, as a function of the grid spacing.}
	\label{fig:convergence}
\end{figure}

\begin{figure}
	\centering
	\includegraphics[width=0.3\textwidth]{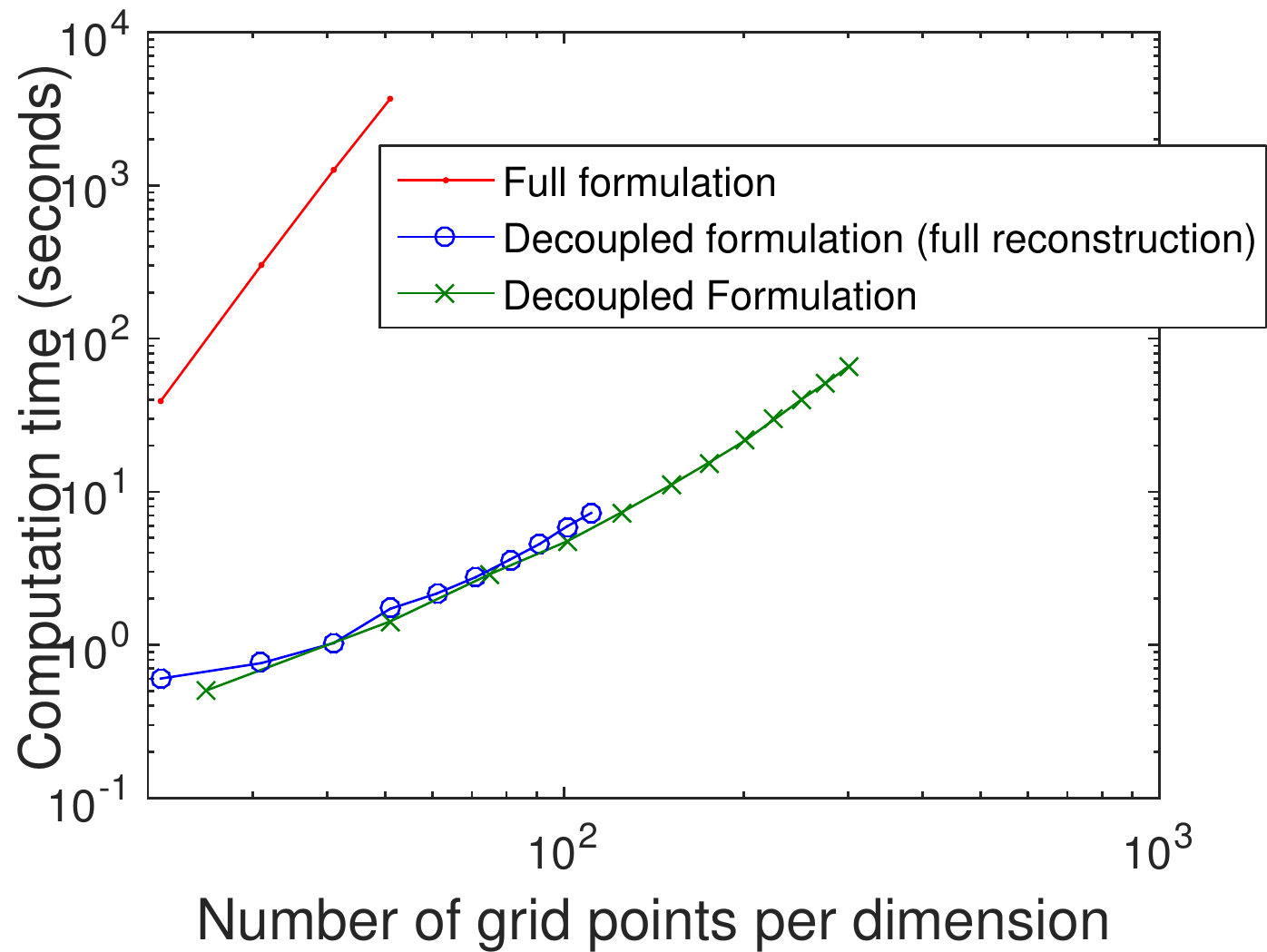}
	\caption{Computation time as a function of the number of grid points in each dimension.}
	\label{fig:compTime}
\end{figure}

\section{6D Quadrotor Collision Avoidance}
Consider relative dynamics augmented by the velocity of evader quadrotor, given in Equation \eqref{eq:rel_dyn_aug}. These dynamics are needed to impose a velocity limit on the quadrotor.

\begin{equation} \label{eq:rel_dyn_aug}
\begin{aligned}
\dot{p}_{x,r} &= v_{x,r}, &\dot{p}_{y,r} &= v_{y,r} \\
\dot{v}_{x,r} &= u_x - d_x, &\dot{v}_{y,r}&= u_y - d_y\\
\dot{v}_{x,1} &= u_x, &\dot{v}_{y,1} &= u_y \\
\end{aligned}
\end{equation}

For this system, we consider a collision between the two quadrotors, as defined previously, to be unsafe configurations. Here, we denote this set $\mathcal{L}_C$

\begin{equation}
\mathcal{L}_C = \{z \in \mathbb{R}^6 \mid |p_{x,r}|\le d, |p_{y,r}|\le 2\}
\end{equation}

\noindent with corresponding implicit surface function $l_C(z)$. 

We also consider configurations in which Player 1 is exceeding a velocity limit of $5$ in the $x$- or $y$- directions to be unsafe. This set of configurations is denoted $\mathcal{L}_S$:
\begin{equation}
\mathcal{L}_S = \{z \in \mathbb{R}^6 \mid |v_{x,1}|\ge5 \vee |v_{y,1}|\ge5\}
\end{equation}

We define the target set $\mathcal L$ to be the union of the above configurations; $\mathcal L$ represents all unsafe configurations. We will compute the reachable set $\mathcal V$ from the target set $\mathcal L$ as follows:

\begin{enumerate}
\item Compute $\mathcal V_C(t)$, the reachable set from target set $\mathcal L_C$.
\item Compute $\mathcal V_S(t)$, the reachable set from target set $\mathcal L_S$.
\item Take the union to obtain $\mathcal V(t) = \mathcal V_C(t) \cup \mathcal V_S(t)$.
\end{enumerate}

\subsection{Reachable set}
Taking the union $\mathcal V=\mathcal V_S(t) \cup \mathcal V_C(t)$, we obtain the 6D reachable set. To visualize $\mathcal V(t)$, we compute 2D slices of the 6D reachable set at various $(v_{x,r}, v_{x,1}, v_{y,r}, v_{y,1})$ values. This is done \textit{without} computing the entire 6D reachable set by setting the computation domain $\mathcal Z$ to be in a large portion of the $(p_{x,r}, p_{y,r})$ plane, at a small range of $(v_{x,r}, v_{x,1}, v_{y,r}, v_{y,1})$ values. The 2D slices at shown in Figure \ref{fig:2x3Dreach}. Each subplot shows two different pairs $(v_{x,1}, v_{y,1})$ for a particular $(v_{x,r}, v_{y,r})$. The red boundary represents the slice with $(v_{x,1}, v_{y,1})=(0,0)$, while the blue boundary represents a slice with the evader velocity almost exceeding the limit of $\bar v = 5$. The blue boundaries contain the red ones, because if the evader is already near the velocity limit, it would have more limited capability to avoid collisions.

\begin{figure}
	\centering
	\includegraphics[width=0.35\textwidth]{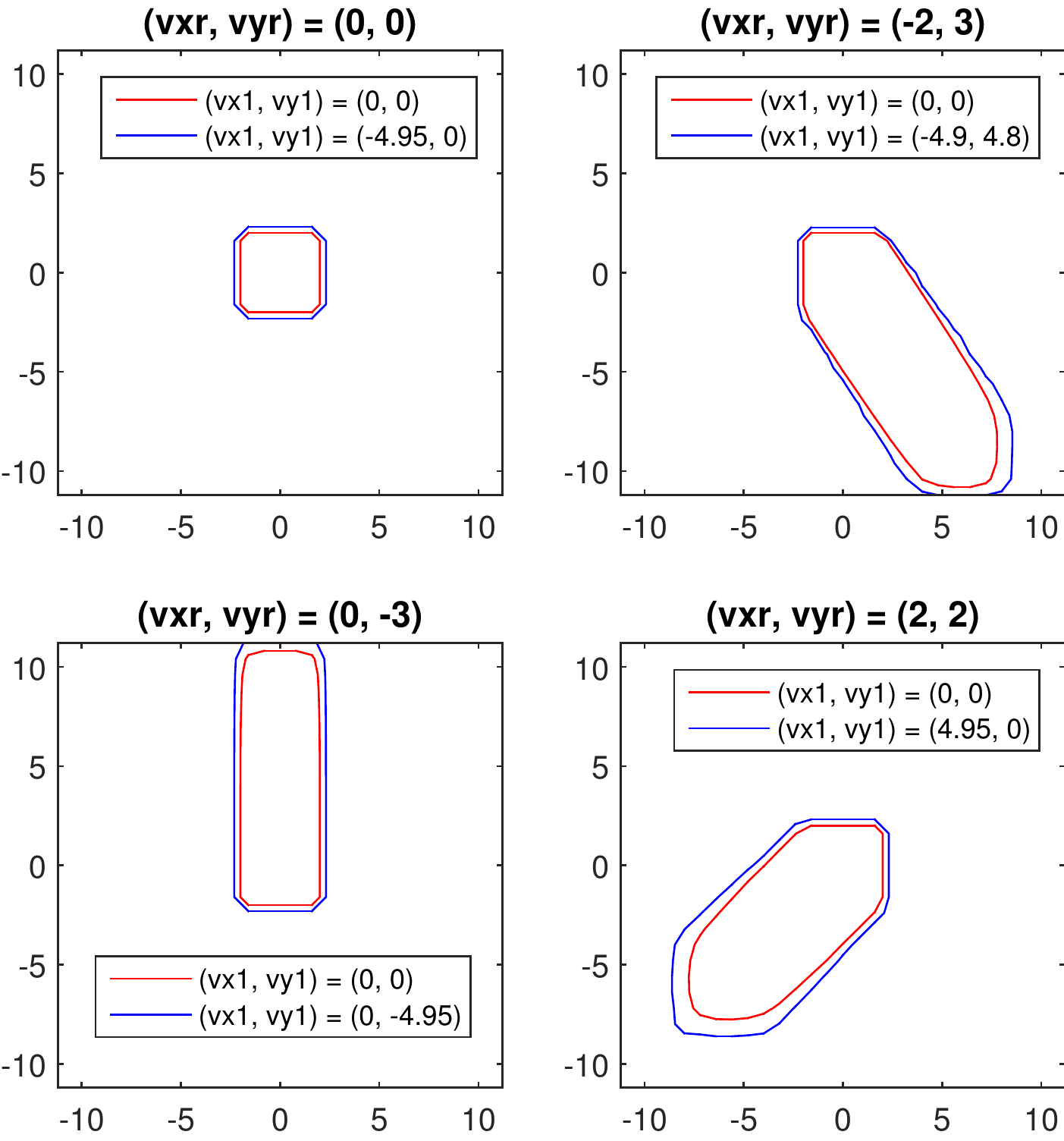}
	\caption{2D slices of the 6D reachable set for the augmented relative dynamics of two quadrotors.}
	\label{fig:2x3Dreach}
\end{figure}

Figure \ref{fig:2x3DreachSim} shows a simulation of the collision avoidance maneuver resulting from the reachable set. The red and blue quadrotors are initially traveling in opposite directions. Consider the situation in which the red quadrotor insists on staying on its intended path. In this case, the blue quadrotor must perform the optimal avoidance control whenever it reaches the boundary of the reachable set ($t=1,3$), shown as the blue dotted boundary. When the blue quadrotor is no longer at the boundary of the reachable set ($t=5$), it is free to perform any control.

%
%
%

\begin{figure}
	\centering
	\includegraphics[width=0.4\textwidth]{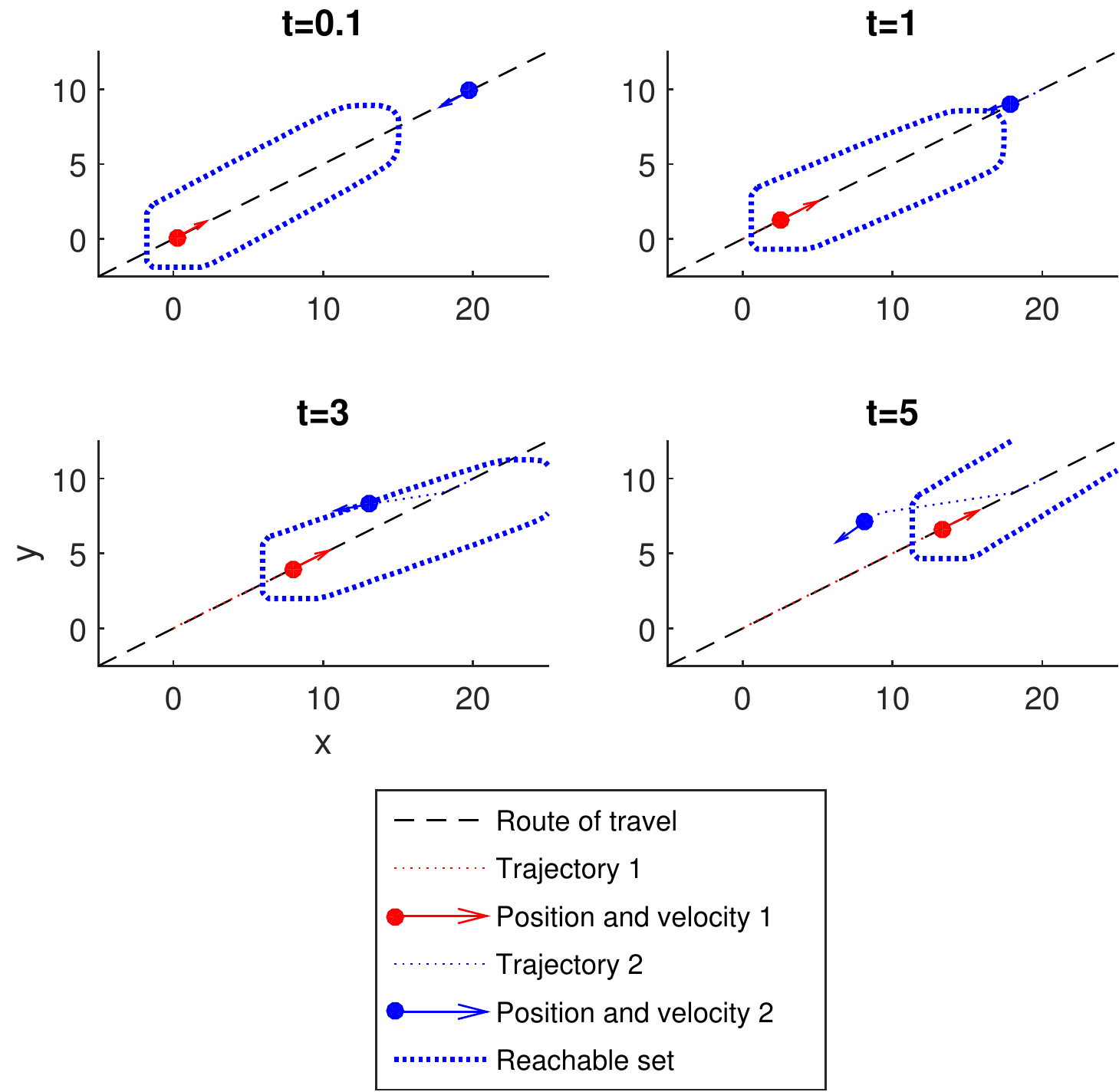}
	\caption{A simulation of collision avoidance between the two quadrotors using the 6D reachable set.}
	\label{fig:2x3DreachSim}
\end{figure}
\section{Conclusions and Future Work} \label{sec:conc}
We have presented a decoupled formulation of the HJ reachability, enabling us to efficiently compute the viscosity solution of the HJ PDE \eqref{eq:HJIPDE}, which gives the reachable set for optimal control problems and differential games. When the system dynamics are decoupled, our decoupled formulation allows the exact reconstruction of the solution to the full dimensional HJ PDE by solving the HJ PDEs corresponding to each of the decoupled components. Our novel approach enables the analysis of otherwise intractable problems. In addition, our formulation achieves this without sacrificing \textit{any} optimality compared to the full formulation. We demonstrated the benefits of our approach using a 4D and 6D quadrotor system.

\bibliographystyle{IEEEtran}
\bibliography{references}

\end{document}